\documentclass{article}
\usepackage{latexsym,amssymb,amsmath,amsthm,amscd,graphicx}

\usepackage[symbol]{footmisc}
\numberwithin{equation}{section}
\newtheorem{theorem}{Theorem}[section]
\newtheorem{lemma}[theorem]{Lemma}

\newtheorem{proposition}[theorem]{Proposition}

\theoremstyle{definition}

\newtheorem{remark}[theorem]{Remark}
\newtheorem{definition}[theorem]{Definition}

\theoremstyle{remark}

\DeclareMathOperator*{\esssup}{ess\,sup}

\usepackage{color}
\definecolor{blue}{rgb}{0,0,0.45}
\definecolor{red}{rgb}{0.7,0,0}

\begin{document}

\begin{center}
\LARGE Some remarks on the maximally modulated  Calder\'on-Zygmund  operator satisfying\\  $L^r$-H\"ormander condition
\end{center}

\

\centerline{{\large Arash Ghorbanalizadeh, Sajjad Hasanvandi}}

\

\begin{abstract}
In this work, by recent work of Lerner and Ombrasi (J. Geom. Anal. 30(1): 1011-1027, 2020), we show a maximally modulated singular integral operator which its kernel satisfying   $L^r$- H\"ormander condition can be dominated by sparse operators. Also, the local exponential decay estimates for these operators are obtained.
\end{abstract}

\

\begin{quote}\small

\end{quote}

\begin{quote}\small
\end{quote}

\vspace{3mm}

\section{Introduction and Preliminaries}

After work of T. Hyt\"onen \cite{HYT}, about the full proof of the $A_2$ conjecture, there were a number of breakthroughs results that have delineated a new theory that we may now call "sparse domination technique". A decisive step toward a modern development of domination by sparse operators was carried out by Andrei Lerner \cite{LER}, which gave an alternative simple proof of the $A_2$ theorem. He showed that Calder\'on--Zygmund operators can be controlled in norm from above by a very special dyadic type operators. Later the pointwise (dual) dominating of an operator is obtained. Since then, sparse bounds for different operators have been a very attractive realm in harmonic analysis. These type dominating of an operator which is typically signed and non-local, by a positive and localized expression of the form \eqref{sparse} has seen an explosion of applications in dyadic harmonic Analysis. The sparse operators are defined in the form
\begin{equation}\label{sparse}
\mathcal A_{p,\mathcal{S}}(f)(x):=\sum_{Q\in \mathcal{S}}  \langle f \rangle_{Q,p} \mathbf 1_{Q}(x),
\end{equation}
where $p\ge 1$, and for any cube $Q$,
$$
\langle f \rangle_{Q,p} := \left(\frac{1}{|Q|} \int_Q |f(x)|^{p}dx \right)^{\frac{1}{p}}.
$$

The purpose of this paper is to present some remarks concerning maximally modulated singular integrals, which it's kernel satisfy $L^r$-H\"ormander condition. These operators are investigated inspire of  Carleson's operator.  Carleson's operator is the modulated Hilbert transform define by
\[
\mathcal{C} f(x)= \sup_{\xi \in \mathbb{R}} |H(\mathcal{M^{\xi}}f)(x)|,
\]
where $\mathcal{M^{\xi}}f(x)= e^{2\pi i \xi x} f(x)$.

The organization of this paper is as follows. The first section of paper is devoted to obtain a sparse bound to  maximally modulated singular integrals and (maximally) modulated maximal singular integral, using known methods especially the works of \cite{NYL} and its improved version \cite{LO}. In non-modulated case, here are some works, for example \cite{HRT, Lac, NYL, LO}, which authors was considering weaker regularity conditions on the kernels of Calder\'on--Zygmund operators.In the second section,  we will consider local exponential decay estimates for the maximally modulated Calder\'on-Zygmund  operator satisfying $L^r$-H\"ormander condition.

 We recall the history of local exponential decay estimate. This type estimate  goes back to work of Coifman and Feierman \cite{CF} which they used the good-$\lambda$ technique introduced by Burkholder and Gundy \cite{BG}. The use of that technique was based on the following estimate
\begin{equation}\label{BAS}
|\{x \in \mathbb{R}^n : T^{\ast} f(x) > 2 \lambda, ~~~ M(f) \le \gamma \lambda \}| \le c \gamma |\{x \in \mathbb{R}^n : T^*(x)f > \lambda\}|
\end{equation}
for any $\lambda > 0$ and for sufficiently small $\gamma > 0$. Here $T^*$ is the maximal singular integral operator of $T$ and  $T^*f(x) := \sup_{\epsilon >0} |T_{\epsilon}f(x)|$. Indeed by the Whitney decomposition, \eqref{BAS} is reduced to a local estimate
\begin{equation}\label{BASW}
|\{x \in \mathbb{R}^n : T^{\ast} f(x) > 2 \lambda, ~~~ M(f) \le \gamma \lambda \}| \le c \gamma |Q|
\end{equation}
where $f$ is supported on $Q$. In 1993, Buckley obtained a local exponential decay in the following form:
\begin{equation}\label{BUCK}
|\{x \in \mathbb{R}^n : T^{\ast} f(x) > 2 \lambda, ~~~ M(f) \le \gamma \lambda \}| \le c e^{\frac{-c}{\gamma}} |Q|
\end{equation}
when he studied  the sharp dependence on the $A_p$ constant of the weight $\omega$ for the operator norm of singular integrals, (see \cite{BUCK}). Buckley proved this estimate using as a model a more classical inequality due to Hunt for the conjugate function which was inspired by a result of Carleson \cite{CARLE}. On the other hand, this exponential decay \eqref{BUCK} has been a crucial step in deriving corresponding sharp $A_1$ estimate in \cite{LEROMCP} and \cite{LEROMCP1}.

We will work with an improved version of inequality \eqref{BUCK} due to Karagulyan \cite{KARA}:
\begin{equation}\label{KARA}
|\{x \in \mathbb{R}^n : T^{\ast} f(x) > t ~ M(f)  \}| \le c e^{\alpha t} |Q|, \qquad t>0.
\end{equation}
We note that the analogue of \eqref{BAS} for the maximally modulated Calder\'on-–Zygmund  operator where $T$ is classical was obtained in \cite{GMS}. So far, there are several methods for obtained local decay in references. We will consider the works \cite{CLPR, OPR}.  The advantage of the method used in \cite{CLPR}, is it can be applied for every operators which has sparse bound.

\subsection{ Notations and basic definitions}

In this article we will be concerned with sparse domination bound for the maximally modulated Calder\'on-Zygmund  operator and then the local exponential decay estimates for such operators will be establish. We say that $T$ is a Calder\'on--Zygmund operator if $T$ is a linear operator of weak type $(1, 1)$ such that
\begin{align}\label{T1 1}
Tf(x) =\int_{ \mathbb{R}^n} k(x,y)f(y)dy \qquad x \notin \texttt{supp} f.
\end{align}
with kernel $K$ satisfying $L^r$-H$\ddot{o}$rmander condition (see definition \ref{H}).

\begin{definition}\label{H}
The kernel $K$ satisfies the $L^r$-H\"ormander condition , $1\leq r \leq \infty $, if

\begin{align*}
\mathcal{\kappa}_{r} := \sup_{Q} ~~\sup_{x',x'' \in \frac{1}{2}Q}~~~ \sum_{k=1}^{\infty} \left| 2^k Q \right|^{\frac{1}{r'}} \left\| K(x',.)- K(x'',.)\right\|_{L^r(2^k Q \setminus 2^{k-1}Q)} < \infty.
\end{align*}
where $r'$ is conjugate exponent of $r$. Denote by $\mathcal{H}_r$ the class of all kernel satisfying the $L^r$-H\"ormander condition.
\end{definition}
Let $\mathcal{H}_{r}$ denote the class of kernels satisfying the  $L^r$-H\"ormander condition. One can see that for any $r>s$, $\mathcal{H}_{r}$ contained in $\mathcal{H}_{s}$ i.e. $\mathcal{H}_{r} \subsetneq \mathcal{H}_{s}$.

\begin{remark}
We remark that $L^r$-H\"ormander condition is weaker than assumption $H_2$ used in \cite[Proposition 3.2]{BCDH}. Then sparse bound obtained in section 2, is valid for modulated of nonstandard kernel operators such as Fourier multipliers and the Riesz transforms associated to Schr\"odinger operators \cite[Remark 3.4]{KL}.  Also $L^r$-H\"ormander condition is weaker than the Dini condition \cite[Proposition 3.1]{KL}. As a result, sparse bound obtained in section 2, is valid for modulated singular integrals which satisfy the Dini's condition.
\end{remark}

In \cite{LO}, Lerner and Ombersi give a method to find sparse bound for Calder\'on-Zygmund operators which it is an improved version of \cite{NYL}. Authors  obtain  nearly minimal assumptions on a singular integral operator $T$ for which it admits a sparse domination. Based on them ideas, we are going to give sparse bound for maximally modulated Calder\'on-Zygmund operators.  Our next definition is the definition  of $W_q$ property of $T$, which is defined in \cite{LO}.
\begin{definition}
We call a sub-linear $T$ satisfy the property of $W_q$, if there exist a non-increasing function  $\psi_{T,q}( \lambda)$ such that for every $Q$ and for every $f \in L^{q}(Q)$,
\begin{align}\label{lo T3 1}
\left|\left\{ x\in Q : |T(f\chi_{Q})(x)| > \psi_{T,q}(\lambda)  \langle f \rangle_{q,Q}  \right\}\right| \leq \lambda |Q| \qquad (0<\lambda<1).
\end{align}
\end{definition}

For definition of maximal singular operator of $T,$ we recall associated with $T$ there is a truncated operator $T_\epsilon$ which is defined as follows:

\begin{align}\label{T2 }
T_{\epsilon}f(x) =\int_{ |x-y|> \epsilon} k(x,y)f(y)dy, \qquad T_{*}f(x)= \sup_{\epsilon >0} |T_{\epsilon}f(x)|.
\end{align}

Definition of Maximally modulated operators $T^{\mathcal{F}}$ inspire of the Carleson operator, is defined as following: let $\mathcal{F}=\{\phi_{\alpha}\}_{\alpha \in A}$ be a family of real-valued measurable functions indexed by some set $A$, and let $T$ be above mentioned operator.
\begin{align*}
T^{\mathcal{F}}f(x):= \sup_{\alpha \in A}\left|T\left(\mathcal{M}^{\phi_{\alpha}}f\right)(x)\right|,
\end{align*}
where $\mathcal{M}^{\phi_{\alpha}}f(x)= e^{2\pi i \phi_{\alpha}(x)}f(x)$. We recall that the Carleson operator is the modulated of Hilbert transform
and the family $F$ consists of the linear functions $\phi_{\alpha}(y) = \alpha y$ with $\alpha \in \mathbb{R}$. We also define the (maximally) modulated maximal singular integral associated with $T$ and $\phi$ via
\begin{equation}\label{Max}
T^{\mathcal{F}}_{\ast}f(x) = \sup_{\epsilon >0} \sup_{\alpha \in A} |T_{\epsilon}(\mathcal{M}^{\phi_{\alpha}}f(x))|.
\end{equation}

We will consider these operators under a priori assumption
\begin{equation}\label{Assum}
\left\|T^{\mathcal{F}}(f)\right\|_{L^{p, \infty}} 	 \lesssim ~~ \psi(p')\left\|f\right\|_{L^{p}}, \qquad 1<p\le 2
\end{equation}
where $\psi$ is a non-decreasing function on $[1, \infty)$ and   $f \in L^1(\mathbb{R}^n)$ with compact support. We remark that inspired by recent approach to the study of the $p=1$ end-point behavior of the Calerson operator via weak $L_p$ bound, we choose to work  under a priori assumption \eqref{Assum}. For more detail information corresponding to make sense of this assumption see \cite{FALon}. Also, we note that the condition \eqref{Assum}, means that this operator satisfy $W_p$ property, i.e. for any $0<\lambda <1$
\begin{equation}
\left|\left\{ x\in Q : |T^{\mathcal{F}}(f\chi_{Q})(x)| > \xi_{T^{\mathcal{F}},q}(\lambda)  \langle f \rangle_{q,Q}  \right\}\right| \leq \lambda |Q| \qquad (0<\lambda<1)
\end{equation}
where $\xi_{T^{\mathcal{F}}, p}(\lambda) := \frac{1}{\lambda \psi^{p}(p')} <f>_{Q}$.

Now we define following quantity
\[
M^{\sharp}_{ T^{\mathcal{F}}, \alpha}f(x) : =  \sup_{Q\ni x} ~\esssup_{x',x" \in Q} \left| T^{\mathcal{F}} \left( f \chi_{\mathbb{R}^n\setminus \alpha Q} \right) (x) -   T^{\mathcal{F}} \left( f \chi_{\mathbb{R}^n\setminus \alpha Q} \right) (x') \right|
\]
where $x,x' \in P$ and $\alpha \ge 3$ .  This quantity is of weak type $(r,r)$ which is proved in Lemma \ref{WEAK}.

Now we recall Orlicz space and some notion related this space which we will use them.

Let $\Phi$ be a Young function, that is, $\Phi : [0,\infty)\to [0,\infty)$, $\Phi$ is continuous, convex, increasing, $\Phi(0) = 0$ and $\Phi(t) \to \infty $ as $t \to \infty$.

Let $f$ be a measurable function defined on a set $Q \subset \mathbb{R}^n$ with finite Lebesgue measure. The $\Phi$-norm of $f$ over $Q$ is defined by
\begin{equation}\label{Orlicz}
  \|f\|_{\Phi,Q}= \inf\left\{\lambda > 0 : \frac{1}{|Q|} \int_{Q} \Phi\left(\frac{|f(x)|}{\lambda}\right)dx \le 1 \right\}.
\end{equation}

The Orlicz maximal operator $M_{\Phi}$ is defined by
\[
M_{\Phi}f(x):= \sup_{Q \ni x} \|f\|_{\Phi,Q}.
\]

Similar to proof of Proposition 6.1 in \cite{FALon}, we can see that for each cube $Q \subset \mathbb{R}^n$ and for $T^{\mathcal{F}}$ maximally modulated Calder\'on-Zygmund operator mentioned above, the condition
\begin{equation}\label{weak type}
\| T^{\mathcal{F}}(f\chi_{Q})\|_{L^{1,\infty}(Q)} \le |Q| \|f\|_{\Phi, Q},
\end{equation}
is sufficient condition for validity of the priori assumption \eqref{Assum}, where the where the Young function $\Phi$ is such that
$$
\gamma_{\Phi}(p) = \sup_{t \ge 1} \frac{\Phi}{t^{p'}} < \infty, \qquad \forall p>1.
$$

\subsection{ Local Mean Oscillation estimate.}

Let $f : Q \to R$ be a measurable function. Here $Q$ could be any set of finite positive measure, but later on it will mostly be a cube; hence the choice of the
letter. The median of $f$ on $Q$ is any real number $m_{f}(Q)$ with the following two properties:
\begin{align*}
\left|Q \cap \left\{f > m_{f}(Q) \right\}\right| &\le \frac{1}{2} |Q|
\\
 \left|Q \cap \left\{f < m_{f}(Q) \right\}\right| &\le \frac{1}{2} |Q|
\end{align*}

The mean local oscillation of a measurable function $f$ on a cube $Q$ is defined by
the following expression
\[
\omega_{\lambda}(f; Q) = \inf_{c\in \mathbb{R}} ((f - c)\chi_{Q})^{*}(\lambda |Q|),
\]
for all $0 < \lambda < 1$, and the local sharp maximal function on a fixed cube $Q_0$ is defined as
\[
M_{\lambda ; Q_0}^{\sharp}f(x)= \sup_{x\in Q \subset Q_0} \omega_{\lambda}(f; Q),
\]
where the supremum is taken over all cubes $Q$ contained in $Q_0$ and such that $x \in Q.$

The decreasing rearrangement concept can be defined for any measurable function $f$. We denote
\[
f^{\ast}(t) := \inf \{\alpha \ge 0 : |\{|f|>\alpha\}| \le t\} \qquad \left(\inf \emptyset:= \infty \right).
\]

We make the following observations (see \cite{HYTL}):
\begin{itemize}
\item
$f^{\ast}$ is non-increasing.

\item
The set inside the infimum is of the form $[a_0,\infty) (\mbox{or} ~~~\emptyset)$. Hence the infimum is reached
as a minimum; in particular, $f^{\ast}$ itself is an admissible value of $\alpha$, so that
$$
|\{|f|>f^{\ast}\}| \le t
$$

\item
We have $(f \chi_{Q})^{\ast}(t) = \inf \{\alpha \ge 0 : |Q \cap \{|f|>\alpha\}| \le t\}$.
\end{itemize}

We will use several times that for any $\delta > 0$, and $0 < \lambda < 1$,
\begin{equation}\label{AV}
(f \chi_{Q})^{\ast} (\lambda |Q|) \le \left(\frac{1}{\lambda |Q|} \int_{Q} |f|^{\delta} dx\right)^{\frac{1}{\delta}}
\end{equation}
this is true since
\begin{equation}\label{OXAY}
(f \chi_{Q})^{\ast}(t) \le \frac{1}{t^{\frac{1}{\delta}}} \|f \chi_{Q}\|_{\delta, \infty}.
\end{equation}

On the other hand, the following lemma is valid which is proved in \cite{HYTL}.
\begin{lemma}\label{ME}
The following estimate holds for all $ \lambda \in (0, \frac{1}{2})$ and all medians $m_{f}(Q)$:
\begin{equation}\label{MI}
|m_{f}(Q)| \le (f \chi_{Q})^{\ast} (\lambda |Q|)
\end{equation}
\end{lemma}

Recall that, for a fixed cube $Q_0, D(Q_0)$ denotes all the dyadic subcubes with
respect to the cube $Q_0$. As before, if $Q \in D(Q_0)$ and $Q \neq Q_0,$ $\hat{Q}$ will be the
ancestor dyadic cube of $Q,$ i.e., the only cube in $D(Q_0)$ that contains $Q$ and such
that $|\hat{Q}| = 2^n |Q|$.

The following theorem was proved by Hyt\"onen \cite[Theorem 2.3]{Hytt} which is improved version of Lerner's formula.  The original version of  Lerner's formula is given in \cite{LER1, LER2}.

\begin{theorem}\label{LF}
For any measurable function $f$ on a cube $Q^0 \subset \mathbb{R}^n,$ we have
\begin{equation}\label{LF}
|f(x)-m_{f}(Q^0)| \le 2 \sum_{L \in \mathcal{L}} \omega_{\frac{1}{2^{n+2}}}(f;L) \chi_{L}(x),
\end{equation}
where $\mathcal{L} \subset D(Q0)$ is sparse: the collection $\mathcal{L} \subset D(Q0)$ is called a $\eta=\frac{1}{2}$-{\it sparse family} of cubes if there exist pairwise disjoint subsets $E_L \subset L$ with $\eta |L|\le  |E_L|$ for each $L \in \mathcal{L}$.
\end{theorem}

\section{Main Results}

The following lemma shows that $M^{\sharp}_{ T^{\mathcal{F}}}$ is of weak type $(r,r)$, where  $1\leq r \leq \infty $.

\begin{lemma}\label{WEAK}
Let $T^{\mathcal{F}}$ be maximally modulated operator defined above. For every  $1\leq r \leq \infty $, the following inequality is satisfied:
\[
M^{\sharp}_{ T^{\mathcal{F}}, \alpha}f(x) \le C M_{r}f(x).
\]
\end{lemma}

\begin{proof}
Let $x_p$ be  the center of cube $P$ and  $P^*= \alpha P$, where $\alpha \ge 3$ then
\begin{align*}
\left| T^{\mathcal{F}}\right. &\left. \left( f \chi_{\mathbb{R}^n\setminus P^*} \right) (x) -   T^{\mathcal{F}} \left( f \chi_{\mathbb{R}^n\setminus P^*} \right) (x')  \right| \\
&\leq \left| T^{\mathcal{F}} \left( f \chi_{\mathbb{R}^n\setminus P^*} \right) (x) - T^{\mathcal{F}} \left( f \chi_{\mathbb{R}^n\setminus P^*} \right) (x_p) \right| + \left| T^{\mathcal{F}} \left( f \chi_{\mathbb{R}^n\setminus P^*} \right) (x') - T^{\mathcal{F}}\left( f \chi_{\mathbb{R}^n\setminus P^*} \right) (x_p) \right|.
\end{align*}
It is sufficient we estimate one of two terms. Let us start with first
\begin{align*}
\left| T^{\mathcal{F}}\right. &\left.\left( f \chi_{\mathbb{R}^n\setminus P^*} \right) (x)-  T^{\mathcal{F}} \left( f \chi_{\mathbb{R}^n\setminus P^*} \right) (x_p) \right|
\\
& = \left| \sup_{\alpha \in A} \left| T \left( \mathcal{M}^{\phi_{\alpha}} f \chi_{\mathbb{R}^n\setminus P^*} \right)(x)\right|  -  \sup_{\alpha \in A}\left|T\left(\mathcal{M}^{\phi_{\alpha}} f \chi_{\mathbb{R}^n\setminus P^*} \right) (x_p) \right| \right|
\\
& \le  \sup_{\alpha \in A} \left| T \left( \mathcal{M}^{\phi_{\alpha}} f \chi_{\mathbb{R}^n\setminus P^*} \right)(x) - T \left( \mathcal{M}^{\phi_{\alpha}} f \chi_{\mathbb{R}^n \setminus P^*} \right) (x_p) \right|,
\end{align*}
\begin{align*}
& \le \int_{ \mathbb{R}^n\setminus P^*} | K(x,y)-K(x_p,y) | |f(y) |dy
\\
& \le \sum_{j=1}^{\infty} \int_{  2^k Q \setminus 2^{k-1}Q} | K(x,y)-K(x_p,y) | |f(y) |dy
\\
& \leq  \sum_{j=1}^{\infty} \left( \int_{  2^k Q \setminus 2^{k-1}Q} | K(x,y)-K(x_p,y) |^{r}dy \right)^{\frac{1}{r}}  \left( \int_{  2^k Q \setminus 2^{k-1}Q} | f(y) |^{r'}dy \right)^{\frac{1}{r'}}   \notag
\\
& \leq \sum_{j=1}^{\infty} \| K(x,.)- K(x_p,.)\|_{L^r(2^k Q \setminus 2^{k-1}Q)}  \left( \frac{|2^k Q|}{|2^k Q|}\int_{  2^k Q } | f(y) |^{r'}dy \right)^{\frac{1}{r'}}  \notag
\\
& \leq  \sup_{Q \ni x} ess\sup_{x',x'' \in Q} \sum_{j=1}^{\infty}  \left| 2^k Q \right|^{\frac{1}{r'}}  \| K(x',.)- K(x'',.)\|_{L^r(2^k Q \setminus 2^{k-1}Q)}  \times  M_{r'} f(x)  \notag
\\
& \leq \mathcal{\kappa}_{r}  ~ M_{r'} f(x).
\end{align*}

Therefore we have
	
\begin{align*}
\left| T^{\mathcal{F}} \left( f \chi_{\mathbb{R}^n\setminus P^*} \right) (x) -  T^{\mathcal{F}} \left( f \chi_{\mathbb{R}^n\setminus P^*} \right) (x') \right| \lesssim \mathcal{\kappa}_{r}  ~~ M_{r'} f(x),
\end{align*}

This argument follows $M^{\sharp}_{ T^{\mathcal{F}}, \alpha}f(x) $ is weak type $(r',r')$.
\end{proof}

Following theorem give sparse bound for maximally modulated singular integral.

\begin{theorem}\label{ModSp}
Let $T^{\mathcal{F}}$ be a sub-linear operator satisfying the $W_q$ condition and such that $\mathcal{M}_{T^{\mathcal{F}},\alpha}^{\#}$ is of weak type $(r,r)$ for some $\alpha \geq 3$, where $1\leq q,r < \infty$. Let $s=\max(q,r)$. Then, for every compactly supported $f \in L^s(\mathbb{R}^n)$, there exists $\frac{1}{2.\alpha^n}$-sparse family $\mathcal{S}$ such that
\[
\left| T^{\mathcal{F}}f(x) \right| \leq C \sum_{Q \in \mathcal{S}} \langle f \rangle_{s,Q} \chi_{Q}(x)
\]
for a.e $x \in \mathbb{R}^n $, where
$C=c_{n,r,s,\alpha} \left(  \psi_{T^{\mathcal{F}},q}(\frac{1}{12.(2\alpha)^n}) + \| \mathcal{M}_{T,\alpha}^{\#}\|_{L^r \to L^{r,\infty}}  \right)$.
\end{theorem}
\begin{proof}
The proof is almost identical to the proof of \cite[Theorem 2.2]{LO}. For the sake of completeness, we give the proof here. For any arbitrary cube $Q$, set $Q^* = \alpha Q$ where $\alpha \ge 3$. In point view of the definition of operator $\mathcal{M}_{T,\alpha}^{\#}$, for every cube $P$ and for all $x',x''\in P$, we have the following inequality
\begin{align} \label{lo T1 2}
\left| T^{\mathcal{F}}\left( f \chi_{\mathbb{R}^n\setminus P^*} \right) (x') -  T^{\mathcal{F}} \left( f \chi_{\mathbb{R}^n\setminus P^*} \right) (x'') \right| \leq \inf_{P} \mathcal{M}_{T^{\mathcal{F}},\alpha}^{\#}(f).
\end{align}

Set
$$
\widetilde{M}_{T^{\mathcal{F}}}f = \max \left( |T^{\mathcal{F}}f|,\mathcal{M}_{T^{\mathcal{F}},\alpha}^{\#} f \right).
$$

In point view of the theorem assumptions, the weak type $(1,1)$ of $M$ and H\"oder's inequality, the set
\begin{align*}
\Omega =\left\{ x\in Q :\max \left( \frac{M_s(f\chi_{Q^*})(x)}{c} , \frac{|\widetilde{M}_{T^{\mathcal{F}}}(f\chi_{Q^*})(x)|}{A }\right) > \langle f \rangle_{s,Q^*}  \right\}
\end{align*}
satisfies $|\Omega| \leq \frac{1}{2^{n+2}} |Q|$, where
\[
A =2\psi_{T^{\mathcal{F}},q}(\frac{1}{12.(2\alpha)^n}) + c_{n,r,s,\alpha}\| \mathcal{M}_{T^{\mathcal{F}},\alpha}^{\#}\|_{L^r \to L^{r,\infty}}.
\]

The local Calder\'on-Zygmund decomposition guarantees that there exists a family of disjoint cubes $\{P_j\}_{j=1}^{\infty} \subset Q $ such that
\begin{equation}\label{lo T1 3}
\frac{1}{2^{n+1}} |P_j| < |P_j \cap \Omega| \leq \frac{1}{2} |P_j|, \qquad \qquad  | \Omega \setminus \cup_{j=1}^{\infty} P_j | = 0 .
\end{equation}
consequently,
\begin{equation} \label{lo T1 4}
|T(f\chi_{Q^*})(x)| \leq A \langle f \rangle_{s,Q^*}  \qquad for \quad a.e  ~x \in Q\setminus \cup_{j=1}^{\infty} P_j .
\end{equation}

On the other hands, for almost all $x\in P_j$ and $x' \in P_j\setminus \Omega$,
\begin{align} \label{lo T1 5}
|T^{\mathcal{F}}(f\chi_{Q^* \setminus P^*_j})(x)| & \leq  \inf_{P_j} \mathcal{M}_{T^{\mathcal{F}},\alpha}^{\#}(f\chi_{P_j^*}) +|T^{\mathcal{F}}(f\chi_{Q^* \setminus P^*_j})(x')| \notag
\\
& =  \inf_{P_j} \mathcal{M}_{T^{\mathcal{F}},\alpha}^{\#}(f\chi_{P_j^*}) + |T^{\mathcal{F}}(f\chi_{Q^* } - f\chi_{P^*_j })(x')| \notag
 \\
& \leq \inf_{P_j} \mathcal{M}_{T^{\mathcal{F}},\alpha}^{\#}(f\chi_{P_j^*}) +  |T^{\mathcal{F}}(f\chi_{Q^*})(x')| + |T^{\mathcal{F}}(f\chi_{P^*_j})(x')| \notag
\\
& \leq \inf_{P_j \setminus \Omega} \mathcal{M}_{T^{\mathcal{F}} ,\alpha} ^{\#}(f\chi_{P_j^*}) + A \langle f \rangle_{s,Q^*} + |T^{\mathcal{F}} (f\chi_{P^*_j})(x')| \notag
\\
& \leq A \langle f \rangle_{s,Q^*} + A \langle f \rangle_{s,Q^*} + |T^{\mathcal{F}}(f\chi_{P^*_j})(x')| \notag
\\
& = 2A \langle f \rangle_{s,Q^*} + |T^{\mathcal{F}}(f\chi_{P^*_j})(x')|.
\end{align}

Thanks to \eqref{lo T1 3}, one can obtain  $ |P_j \setminus \Omega| \geq \frac{1}{2} |P_j|$. On the other hand,
\begin{equation*}
| \Omega' | = \left| \left\{ x\in P_j: |T^{\mathcal{F}}(f\chi_{P_j^*})(x)| >A \langle f \rangle_{s,P_j^*}  \right\} \right| \leq \frac{1}{2^{n+2}} \ | P_j |.
\end{equation*}

As result we have
\[
\inf_{P_j \setminus \Omega}|T^{\mathcal{F}}(f\chi_{P_j^*})| \leq A \langle f \rangle_{s,P^*}  \leq A \inf_{P_j^*} M_sf \leq A \inf_{P_j} M_sf \leq A \inf_{P_j \setminus \Omega} M_sf \leq c A \langle f \rangle_{s,Q^*},
\]
which, combined with \eqref{lo T1 5}, implies that for all $x \in P_j$,
\begin{equation}\label{lo T1 6}
|T^{\mathcal{F}}(f\chi_{Q^* \setminus P^*_j})(x)| \leq (2+c)A \langle f \rangle_{s,Q^*}.
\end{equation}

Form this and from \eqref{lo T1 4}, for  a.e $x \in Q$

\begin{align} \label{lo T1 7}
|T^{\mathcal{F}}(f\chi_{Q^*})(x)|\chi_{Q}(x)& = |T^{\mathcal{F}}(f\chi_{Q^*})(x)| \chi_{Q \setminus \cup_{j=1}^\infty P_j}(x) +  |T^{\mathcal{F}}(f\chi_{Q^*})(x)| \chi_{ \cup_{j=1}^\infty P_j}(x) \notag
\\
& = (3+c)A \langle f \rangle_{s,Q^*} + \sum_{j=1}^{\infty} |T^{\mathcal{F}}(f\chi_{P_j^*})(x)|\chi_{ P_j}(x).
\end{align}

By \eqref{lo T1 3}, $\sum_{j=1}^{\infty}  |P_j| \leq \frac{1}{2} |Q|.$ Therefore, iterating \eqref{lo T1 7}, we obtain a
$\frac{1}{2}$-sparse family $\mathcal{F}_Q$ of sub cubes of $Q$ such that
\begin{equation} \label{lo T1 8}
|T^{\mathcal{F}}(f\chi_{Q^*})(x)| \leq (3 +c) A   \sum_{\mathcal{R} \in \mathcal{F}_Q} \langle f \rangle_{s,\mathcal{R}^*}\chi_{\mathcal{R}}(x).
\end{equation}
So, the proof is completed with \cite[Lemma 2.1]{LO}.
\end{proof}

\begin{remark}
The cubes of the resulting sparse family $S$ are not dyadic. But there is a well known result which says for an arbitrary cube $Q$, there are $n+1$ general dyadic grids $D^{\alpha}$ such that every cube $ Q \subset \mathbb{R}^n$ is contained in some cube $Q' \in D^{\alpha}$ such that $|Q| \le c_n |Q'|$ (see \cite{Conde}). So, in Theorem \ref{ModSp}, one can write
\begin{equation}\label{BDS}
\left| T^{\mathcal{F}}f(x) \right| \leq C_{n,s} \sum_{j=1}^{n+1} \sum_{Q \in \mathcal{S}_j} \langle f \rangle_{s,Q} \chi_{Q}(x),
\end{equation}
where $\mathcal{S}_j$ is a sparse family from a dyadic grid $D^j$.  We note that in most of papers number of dyadic grids assumed $3^n$ or $2^n$, for example see the papers \cite{LNE, HYT3,GJP}, but the number $n+1$ is optimal.
\end{remark}

Readily we have following theorem regarding sparse bound to   $T^{\mathcal{F}}_{\ast}f(x)$ which defined in \eqref{Max}.
\begin{proposition}\label{MMO}
Let $T^{\mathcal{F}}$ be a sub-linear operator satisfying the $W_q$ condition and such that $\mathcal{M}_{T^{\mathcal{F}},\alpha}^{\#}$ is of weak type $(r,r)$ for some $\alpha \geq 3$, where $1\leq q,r < \infty$. Let $s=\max(q,r)$. Then, for every compactly supported $f \in L^s(\mathbb{R}^n)$, there exists $\frac{1}{2.\alpha^n}$-sparse family $\mathcal{S}$ such that
\begin{equation}\label{BDS}
\left|T^{\mathcal{F}}_{\ast}f(x) \right| \leq C_{n,s} \sum_{j=1}^{n+1} A_{S_j}^sf(x) ,
\end{equation}
where $A_{S_j}^sf(x) = \sum\limits_{Q \in \mathcal{S}_j} \langle f \rangle_{s,Q} \chi_{Q}(x) $.
\end{proposition}

\begin{theorem}
Let $T^{\mathcal{F}}$ be a maximally modulated singular integral operator with the (maximally) modulated maximal singular integral $T^{\mathcal{F}}_{\ast}$. Let $Q$ be a cube and let $f \in L^{\infty}_{c}(\mathbb{R}^{n})$ such that $\texttt{supp}(f) \subseteq Q$. Then there are constants $\alpha, c > 0$ such that
\begin{equation}
|\{ x \in Q : |T^{\mathcal{F}}_{\ast} f(x)| > t~ M_rf(x)\}| \le c e^{-at} |Q|, \quad  t > 0.
\end{equation}
\end{theorem}

\begin{proof}
It follows from Lemma \ref{ME}, \eqref{AV}, \eqref{weak type} and Kolmogorov's inequality
\begin{align*}
|m_{T^{\mathcal{F}}}(Q)| &\le \left(\frac{4}{ |Q|} \int_{Q} |T^{\mathcal{F}}|^{\delta} dx\right)^{\frac{1}{\delta}},
\\
&\le  c_{\delta}\|T^{\mathcal{F}}(f)\|_{L^{1, \infty}(Q, \frac{dx}{|Q|})}
\\
& \le c_{\delta} ~ \|f(x)\|_{\Phi,Q}
\\
& \le c_{\delta} ~  M_{\Phi}(f)(x)
\\
&\le c_{\delta} ~ \gamma_{\Phi}(r')^{\frac{1}{r}} M_{r}(f)(x)
\end{align*}
where $r >1$ and for any $0<\delta$. In last inequality we use the \cite[Lemma 6.3]{FALon}.
By  consequence of the definitions of following notions, for given a cube $Q$, $ \delta> 0$ and $0 < \lambda < \frac{1}{2},$ there exists a constant $c = c_{\lambda}$
such that
\begin{equation}\label{Key}
M_{\lambda ; Q}^{\sharp}(f\chi_{Q})(x) \le c M_{\delta}^{\sharp}(f\chi_{Q})(x), \qquad x \in Q.
\end{equation}
where
$M_{\delta}^{\sharp}(f\chi_{Q})(x)= \sup\limits_{Q \ni x} \inf\limits_{c} \left(\frac{1}{|Q|} \int_{Q} |f(y)-c|^{\delta}\right)^{\frac{1}{\delta}}$. This inequality, readily valid as following:
\begin{align*}
M_{\lambda ; Q}^{\sharp}f(x)&= \sup_{ Q \ni x } \omega_{\lambda}(f; Q)
\\
&=  \sup_{Q \ni x} \inf_{c\in \mathbb{R}} ((f - c)\chi_{Q})^{*}(\lambda |Q|),
\\
& \le c_{\lambda} \sup_{Q \ni x} \inf_{c\in \mathbb{R}} \left(\frac{1}{|Q|}\int_{Q} |f(y) - c|^{^{\delta} } dy \right)^{\frac{1}{^{\delta} }}
\\
& \le c_{\lambda} M_{\delta}^{\sharp}(f\chi_{Q})(x)
\end{align*}
the last inequality comes from $(f \chi_{Q})^{\ast}(t) \le \frac{1}{t^{\frac{1}{\delta}}} \|f \chi_{Q}\|_{\delta, \infty}$.

Also, based on arguments in \cite[proposition 6.1]{FALon} and \cite[Lemma 6.3]{FALon}, one can observe
\[
M_{\delta}^{\sharp}(T^{\mathcal{F}} f)(x) \lesssim   M_{\Phi}f(x) \lesssim M_{r}f(x).
\]
Moreover, the following Feierman-Stein inequality  was obtained in \cite{OPR}:
\[
|\{x \in Q; |f(x)-m_{Q}(f)| > t M^\sharp_{\lambda; Q}(f)(x)\}| \le c_3 e^{\beta t } |Q|,  \qquad \lambda = \frac{1}{2^{n+2}}.
\]

Consequently, we deduce that for $t > c_1$
\begin{align*}
& \left|\left\{x \in Q; ~|T^{\mathcal{F}}_{\ast}f(x)| > t M_{r}(f)(x)\right\}\right| \le \left|\left\{x \in Q; ~|T^{\mathcal{F}}f(x)| > t M_{r}(f)(x)\right\}\right|
\\
&\qquad \qquad \qquad\le\left|\left\{x \in Q; ~|T^{\mathcal{F}}f(x) - m_{T^{\mathcal{F}}f}(Q)| > (t-c_1) M_{r}(f)(x)\right\}\right|
\\
&\qquad \qquad \qquad\le\left|\left\{x \in Q; |T^{\mathcal{F}}-m_{T^{\mathcal{F}}f(x)}(Q)| > c_2^{-1} (t-c_1) M^{\sharp}_{p; 2^{-2-n},Q}(T^{\mathcal{F}})f(x)\right\}\right|
\\
&\qquad \qquad \qquad \le c_3e^{-\frac{\beta(t-c_1)}{c_2}} |Q|.
\end{align*}
So, taking $c = \max \{1, c_3\}  e^{\frac{\beta ~ c_1}{c_2}}$ and $\alpha = \frac{\beta}{c_2}$, we obtain the desired result.

\end{proof}

\begin{remark}
We remark that local decay estimate can be derive by the combination of the sparse domination results obtained in Theorem \ref{ModSp} and Proposition \ref{MMO} and the estimate
\[
 \left|\left\{x \in Q: \mathcal{A}_{S}^{r}f > t M_{r}(f) \right\}\right| \le c_1 e^{c_2 t^{r}} |Q|, \qquad r>0
\]
which is proved in  \cite{CLPR}.
\end{remark}

\bibliographystyle{amsplain}

\bigskip

\textbf{Arash Ghorbanalizadeh}

Department of Mathematics,

Institute for Advanced Studies in Basic Sciences (IASBS), Zanjan 45137-66731, Iran

E-mail: ghorbanalizadeh@iasbs.ac.ir; \,\,\, gurbanalizade@gmail.com

\bigskip

\textbf{Sajjad Hasanvandi}

Department of Mathematics,

Institute for Advanced Studies in Basic Sciences (IASBS), Zanjan 45137-66731, Iran

E-mail: hasanvandi.s@iasbs.ac.ir.

\end{document}